\documentclass[letterpaper,11pt]{article}

\usepackage{amsthm}
\usepackage{amsmath}

\newtheorem{thm}{Theorem}[section]
\newtheorem{Lemma}[thm]{Lemma}

\newtheorem{prop}[thm]{Proposition}

\newtheorem{conj}[thm]{Conjecture}

\newtheorem{question}[thm]{Question}

\newcommand{\beq}[1]{\begin{equation}\label{#1}}
\newcommand{\enq}[0]{\end{equation}}

\newcommand{\bn}[0]{\bigskip\noindent}
\newcommand{\mn}[0]{\medskip\noindent}
\newcommand{\nin}[0]{\noindent}

\newcommand{\sub}[0]{\subseteq}
\newcommand{\sm}[0]{\setminus}
\renewcommand{\dots}[0]{,\ldots,}

\newcommand{\ov}[0]{\overline}

\newcommand{\B}[0]{{\cal B}}
\newcommand{\cee}[0]{{\cal C}}

\newcommand{\f}[0]{{\cal F}}

\newcommand{\h}[0]{{\cal H}}

\newcommand{\pee}[0]{{\cal P}}

\newcommand{\U}[0]{{\cal U}}
\newcommand{\X}[0]{{\cal X}}
\newcommand{\Y}[0]{{\cal Y}}

\newcommand{\Ra}[0]{\Rightarrow}
\newcommand{\lra}[0]{\leftrightarrow}

\newcommand{\xx}[0]{t}

\newcommand{\E}[0]{{\sf E}}
\newcommand{\0}[0]{\emptyset}

\renewcommand{\qed}[0]{\begin{flushright} \rule{2mm}{3mm} \end{flushright}}
\def\qqed{\null\nobreak\hfill\hbox{${\diamondsuit}$}\par\smallskip}

\newcommand{\C}[0]{\binom}
\newcommand{\ga}[0]{\alpha }
\newcommand{\gb}[0]{\beta }
\newcommand{\gc}[0]{\gamma }
\newcommand{\gd}[0]{\delta }
\newcommand{\gD}[0]{\Delta }
\newcommand{\gG}[0]{\Gamma }
\newcommand{\gk}[0]{\kappa }

\newcommand{\go}[0]{\omega}
\newcommand{\gO}[0]{\Omega}

\newcommand{\gs}[0]{\sigma}

\newcommand{\eps}[0]{\varepsilon }

\newcommand{\covby}[0]{<\hspace{-0.09in}\cdot~}

\newcommand{\uk}[0]{\underline{k}}

\begin{document}

\renewcommand{\thefootnote}{\fnsymbol{footnote}}
\title{On the Number of 2-SAT Functions\footnotemark }

\author{
L. Ilinca and J. Kahn}
\date{}

\footnotetext{ * Supported by NSF grant DMS0701175.}

\maketitle

\begin{abstract}
We give an alternate proof of a conjecture of Bollob\'as, Brightwell and
Leader, first proved by Peter Allen, stating that
the number of boolean functions definable  by
2-SAT formulae is $(1+o(1))2^{\C{n+1}{2}}$.
One step in the proof determines the asymptotics of the number of
``odd-blue-triangle-free" graphs on $n$ vertices.
\end{abstract}

\section{Introduction}

Let $\{x_1,\ldots,x_n\}$ be a collection of Boolean variables. Each variable $x$ is associated with a ${\it positive}$ literal, $x$, and a ${\it negative}$ literal, $\bar{x}$. Recall that a {\it k-SAT formula} is an expression of the form
\beq{kS}
C_1\wedge \cdots \wedge C_t,
\enq
with $t$ a positive integer and each $C_i$ a {\em k-clause}; that is, an expression
$y_1\vee\cdots \vee y_k$, with $y_1\dots y_k$ literals corresponding to
different variables.
A formula (\ref{kS}) defines a Boolean function of $x_1\dots x_n$ in the obvious way;
any such function is a {\em k-SAT function}.
Here we will be concerned almost exclusively with the case $k=2$,
and henceforth write ``clause" for ``2-clause."

We are interested in the number of 2-SAT functions of $n$ variables,
which, following \cite{BBL}, we denote $G(n)$.
Of course $G(n)$ is at most
$\exp_2[4\C{n}{2}]$,
the number of 2-SAT formulae; on the other hand it's
easy to see that
\beq{2S}
G(n) > 2^n(2^{\C{n}{2}}-1) \sim 2^{\C{n+1}{2}}
\enq
(all formulas gotten by choosing
$y_i\in \{x_i,\ov{x}_i\}$ for each $i$
and a nonempty subset of the clauses using $y_1\dots y_n$ give different
functions).

The problem of estimating $G(n)$ was suggested by
Bollob\'as, Brightwell and Leader \cite{BBL}
(and also, according to \cite{BBL}, by U. Martin).
They showed that
\beq{bbl}
G(n)=\exp_2[(1+o(1))n^2/2],
\enq
and made the natural conjecture that (\ref{2S}) gives the
asymptotic value of $G(n)$; this was proved in \cite{Allen}:
\begin{thm}
\label{SATthm}
\begin{equation}
\label{SAT}
G(n)=(1+o(1))2^{\C{n+1}{2}}
\end{equation}
\end{thm}

\nin
Here we give an alternate proof.  An interesting feature of our argument
is that it follows the original colored graph approach of \cite{BBL}, in
the process determining (Theorem \ref{OBTFthm}) the asymptotics of the number of
``odd-blue-triangle-free" graphs on $n$ vertices;
both \cite{BBL} and \cite{Allen} mention the seeming difficulty of
proving Theorem \ref{SATthm} along these lines.

The argument of \cite{BBL} reduces (\ref{bbl}) to estimation of the number of
``odd-blue-triangle-free" (OBTF) graphs (defined below).
In brief, with elaboration below, this goes
as follows.
Each ``elementary" 2-SAT function (non-elementary functions are easily disposed of)
corresponds to an OBTF graph; this correspondence is not injective, but the number
of functions mapping to a given graph is trivially $\exp[o(n^2)]$, so that
a bound $\exp_2[(1+o(1))n^2/2]$ on the number, say $F(n)$,
of OBTF graphs on $n$ vertices---proving which is the main occupation
of \cite{BBL}---gives (\ref{bbl}).

The Bollob\'as {\em et al.}
reduction to OBTF graphs is also the starting point for the
proof of Theorem \ref{SATthm}, and a
second main point here (Theorem \ref{OBTFthm}) will be
determination of the asymptotic behavior of $F(n)$.
Note, however, that derivation of Theorem \ref{SATthm} from this is---in contrast to
the corresponding step in \cite{BBL}---not at all straightforward,
since we can no longer afford a crude bound on the number of 2-SAT functions
corresponding to a given OBTF $G$.

It's natural to try to attack the problem of (approximately) enumerating OBTF graphs
using ideas from the large literature on asymptotic enumeration in the spirit
of \cite{EKR}, for instance \cite{KR} and \cite{PS}.
This is suggested in \cite{BBL}; but the authors
say their attempts in this direction were not successful, and their
eventual treatment of $F(n)$
is based instead---as is Allen's proof of Theorem \ref{SATthm}---on
the Regularity Lemma of Endre Szemer\'edi \cite{Szemeredi}.
Here our arguments will be very much in the spirit of the papers mentioned;
\cite{PS} in particular was helpful in providing some initial inspiration.
We now turn to more precise descriptions.

\bigskip
We consider {\em colored}
graphs, meaning graphs with edges colored {\em red} ($R$) and {\em blue} ($B$).
For such a graph $G$, a subset of $E(G)$ is
{\it odd-blue} if it has an odd number of {\it blue} edges (and {\em even-blue}
otherwise),
and (of course) $G$ is {\it odd-blue-triangle-free} (OBTF) if it contains
no odd-blue triangle.
We use $\mathcal{F}(n)$ for the set of (labelled) OBTF graphs on $n$ vertices
and set $|\f(n)| = F(n)$.

A graph $G$ (colored as above) is
{\it blue-bipartite} (BB) if there is a partition $U\sqcup W$ of $V(G)$
such that each blue edge has one endpoint in each of $U$, $W$, while any red
edge is contained in one of $U$, $W$.
We use $\mathcal{B}(n)$ for the set of blue-bipartite graphs on $n$ vertices.
%
It is easy to see that
\begin{equation}
\label{BB}
|\mathcal{B}(n)|=(1-o(1))2^{\C{n+1}{2}-1}.
\end{equation}
(The term $\exp_2[\C{n+1}{2}-1]=\exp_2[n-1+\C{n}{2}]$
counts ways of choosing the unordered pair
$\{U,W\}$ and an {\em uncolored} $G$, the coloring then being dictated
by ``blue-biparticity";
that the r.h.s. of (\ref{BB}) is a {\em lower} bound follows from the observation
that almost all such choices will have $G$ connected, in which case
different $\{U,W\}$'s give different colorings.)

As mentioned above, the main step in the proof of (\ref{bbl})
in \cite{BBL}
 was
a bound $F(n) < \exp_2[(1+o(1))n^2/2]$; here we prove
the natural conjecture that most OBTF graphs are blue-bipartite:

\begin{thm}
\label{OBTFthm}
\begin{equation}
\label{OBTFtoBB}
F(n)=(1+o(1))2^{\C{n+1}{2}-1}
\end{equation}
\end{thm}

\noindent
The bound here corresponds to that in Theorem \ref{SATthm},
in that (as explained below) one expects a typical OBTF $G$ to correspond to
exactly two 2-SAT functions.
Proving that this is indeed the case, and
controlling the contributions of those
OBTF $G$'s for which the number is larger, are the
main concerns of Sections 4 and 5 (which handle blue-bipartite and
non-blue-bipartite $G$ respectively).
These are preceded by a review, in
Section 2, of the reduction from 2-SAT functions to OBTF graphs,
and, in Section 3, the proof of
Theorem \ref{OBTFthm} in a form that gives some further limitations on
graphs in $\f(n)\sm \B(n)$.
The end of the proof of Theorem \ref{SATthm} is given in Section 6,
and Section 7 contains a few additional remarks and questions.

\mn
{\em Numerical usage.}
We use $[n]$ for $\{1\dots n\}$ and
$\C{n}{<k}$ for $\sum_{i=0}^{k-1}\C{n}{i}$.
All logarithms and the entropy $H(\cdot)$ are binary.
We pretend throughout that large numbers are integers.

\mn
{\em Graph-theoretic usage.}
We use
$\Gamma_x$ or $\Gamma(x)$ for the neighborhood of a vertex $x$,
preferring the former but occasionally resorting to the latter
for typographical reasons
(to avoid double subscripts or because we need the subscript to
specify the graph).
For a set of vertices $Q$, $\Gamma(Q)$ is $\cup_{x\in Q}\Gamma_x\sm Q$.
We use $\nabla(X,Y)$ for the set of edges having one end in $X$
and the other in $Y$ ($X$ and $Y$ will usually be disjoint,
but we don't require this).

\section{Reduction to OBTF graphs}

In this section we recall what we need of the reduction
from 2-SAT functions to OBTF graphs, usually
referring to \cite{BBL} for details.

The {\it spine} of a non-trivial 2-SAT function $S$ is the set of variables that
take only one value (True or False) in
satisfying assignments for $S$.
For a 2-SAT function $S$ with empty spine, we say that variables $x,y$ are
{\it associated} if either $x\Leftrightarrow y$ is True in all satisfying
assignments for $S$, or $x\Leftrightarrow \bar{y}$ has this property.
A 2-SAT function with empty spine and no associated pairs is {\it elementary}.
As shown in \cite{BBL},
the number, $H(n)$, of elementary, $n$-variable 2-SAT function satisfies
$$
H(n)\leq G(n)\leq 1+\sum_{k=0}^{n}H(n-k)\C{n}{k}(2n-2k+2)^k,
$$
and it follows that for Theorem \ref{SATthm}
it is enough to show
\begin{equation}
\label{SATe2}
H(n)=(1+o(1))2^{\C{n+1}{2}}.
\end{equation}

Given a 2-SAT formula $F$ giving rise to an elementary function $S_F$, we
construct a partial order $P_F$ on
$\{x_1,\bar{x}_1,\ldots,x_n,\bar{x}_n\}$,
by setting $x<y$ if $\bar{x}\vee  y$ appears in $F$
(so $x\Rightarrow y$ is True in any satisfying assignment for $F$;
note $x,y$ can be positive or negative literals),
and taking the transitive closure of this relation. Then $P_F$ is
indeed a poset and satisfies

\begin{enumerate}
\item [(a)] $P_F$ depends only on the function $S_F$,
\item [(b)] each pair $x$, $\bar{x}$ is incomparable, and
\item [(c)] $x<y$ if and only if $\bar{y}<\bar{x}$.
\end{enumerate}
This construction turns out to give a bijection between the set of
elementary 2-SAT functions and the set $\mathcal{P}(n)$ of posets
on $\{x_1,\bar{x}_1,\ldots,x_n,\bar{x}_n\}$
satisfying (b) and (c),
and in proving (\ref{SATe2}) we work with
the interpretation $H(n)=|\mathcal{P}(n)|$.

For $P\in \mathcal{P}(n)$ we construct a colored graph $G(P)$ on (say) vertex set $\{w_1,\ldots,w_n\}$ by including a {\it red} edge $w_iw_j$ whenever
$x_i\covby\bar{x}_j$ or $\bar{x}_i\covby x_j$ in $P$
(where, as usual, $x\covby y$ means $x<y$ and there is no $z$ with $x<z<y$),
and a {\it blue} edge $w_iw_j$ whenever $x_i\covby x_j$ or
$\bar{x}_i\covby \bar{x}_j$.
Then

\begin{enumerate}
\item[(d)] no edge of $G(P)$ is colored both red and blue,
\item[(e)] $G(P)$ determines the $cover\ graph$ of
$P$
(the set of pairs $\{x,y\}$ for which $x\covby y$ or $y\covby x$), and
\item[(f)] $G(P)\in \mathcal{F}(n)$.
\end{enumerate}
Of course (e) is not enough to get us from
Theorem \ref{OBTFthm}
to the desired
bound (\ref{SATe2}) on $H(n)$ ($=|\mathcal{P}(n)|$),
since it may be that a given cover graph corresponds to many $P$'s.
It turns out
that a typical blue-bipartite $G$ does give rise to exactly
two $P$'s; but bounding the contributions of
general $G$'s is not so easy, and, {\em inter alia}, will
require a somewhat stronger version of Theorem \ref{OBTFthm}
(Theorem \ref{MT}).
If we only wanted Theorem \ref{OBTFthm}, then
Section \ref{SNBB} could be simplified, though the basic
argument would not change.

\section{Nearly blue-bipartite}\label{SNBB}

Fix
$C$ and $\eps>0$.  We won't bother giving these numerical values.
We choose $\eps $ so that the expression on the
right hand side of (\ref{c}) is less than 2,
let $c < 1-0.6\log_23$ be some positive constant satisfying (\ref{c}),
and choose (say)
\beq{C}
C> 12/c.
\enq
Set $s=s(n) = C\log n$.

Throughout the following discussion, $G$ is assumed to lie in $\f(n)$
and we use $V$ for $[n]$, the common vertex set of these $G$'s.
Set
$
\gk(G)=\min\{|K|:K\sub V, G-K ~\mbox{is blue-bipartite}\}.
$
Our main technical result is

\begin{thm}\label{MT}
There is a constant $c>0$ such that for sufficiently large n and
any $t \leq s$,
\beq{mainbd}
|\{G:\gk(G)\geq t\}| <  
2^{(1-c)sn}F(n-s)+
2^{(1-c)n\lceil t/3\rceil}F(n-\lceil t/3\rceil).
\enq
\end{thm}

\nin
Notice that, according to  (\ref{OBTFtoBB}), we expect
$F(n)\approx 2^{na}F(n-a)$ (for $a$ not too large);
so (\ref{mainbd}) says that non-BB graphs contribute little to this growth.
The easy derivation of Theorem \ref{OBTFthm}
from Theorem \ref{MT} is given near the end of this section.

Very roughly, the proof of Theorem \ref{MT} proceeds by
identifying several possible ways in which a graph might
be anomalously sparse
(see Lemmas \ref{Xdoi}-\ref{LX3} and \ref{LX5}),
and showing that graphs with many anomalies are rare,
while for those with few, $\gk$ is small.
Central to our argument will be our ability to say that for
most $G$ and most vertices $x$, there is a small (size about $\log n$)
subset of $\Gamma_x$ whose neighborhood is most of $G$.
The next lemma is a first step in this direction.

Let
$$\X_0(n,t)= \{G\in\mathcal{F}(n):\ \exists\; Q\subseteq V(G) \mbox{ with
$|Q|=t$ and $|\Gamma(Q)|< 0.6 n $}\}$$
and $\X_0=\X_0(n) = \X_0(n,s)$.

\begin{Lemma}
\label{Xdoi}

For sufficiently large n and $s\geq t> \go(1)$,
$$
|\X_0(n,t)| <
2^{(0.6 \log_23 +o(1))tn}  F(n-t)
$$
\end{Lemma}

\nin
{\em Remark }  
The statement is actually valid as long as $t<o(n)$,
but we will only use it with $t=s$.
In place of 0.6 we could use any constant
$\ga$ with $\ga \log_23<1$ and
$\ga > 1/2$, the latter being
crucial for Lemma \ref{LX3}.

\begin{proof}
All $G\in \mathcal{X}_0(n,t)$ can be constructed by choosing:
$Q$; $G-Q$; $\Gamma(Q)$; and the restriction of $G$
(including colors)
to edges meeting $Q$ (where we require $|Q|=t$ and $|\Gamma(Q)| < 0.6 n$).
We may bound the numbers of choices for these steps by (respectively):
$\C{n}{t}$; $F(n-t)$; $2^{n-t}$;
and $\exp_3[\C{t}{2} + t(0.6 n)]$.
The lemma follows.
\end{proof}

Set
$K_1(G) = \{x\in V(G): d_G(x) < \eps n\}$, $\gk_1(G)=|K_1(G)|$,
and
for $t\leq n$,
$$ \X_1(n,t) = \{G\in \f(n): \gk_1(G)\geq t\}.$$
Set $\X_1=\X_1(n) = \X_1(n,s)$.

\begin{Lemma}
\label{Xunu}
For sufficiently large n and any t,
$$
|\X_1(n,t)| < 2^{(H(\eps)+\eps +o(1))nt} F(n-t),
$$
\end{Lemma}

\nin
\begin{proof}
All $G\in \X_1(n,t)$ can be constructed by choosing:
some $t$-subset $K$ of
$K_1(G)$; $G-K$; and $\Gamma_x$
and colors for $\nabla(x,\gG_x)$ for each $x\in K$.
(Of course redundancies here and in similar arguments later
only help us.)
The numbers of choices for these steps are bounded by:
$\C{n}{t}$; $F(n-t)$;
and $(\C{n}{< \eps n}2^{\eps n})^t $.
The lemma follows.
\end{proof}

For each $G\in \f(n)$, let
$K_2(G) = \{x_1\dots x_l,y_1\dots y_l\}$ be a largest possible
collection of (distinct) vertices of $V\sm K_1(G)$
such that
$|\Gamma(x_i)\cap\Gamma(y_i)| <\eps  n$
$\forall~i\in [l]$, and set $\gk_2(G)=l$.
Let
$$
 \X_2(n,t) = \{G\in \f(n)\sm (\X_0(n)\cup \X_1(n)):
 \gk_2(G) \geq t\}
$$
and $\X_2 = \X_2(n) = \X_2(n,s)$.
The next lemma is perhaps our central one.

\begin{Lemma}\label{LX3}
For sufficiently large n and $ t<o(n)$,
$$|\X_2(n,t)| <
2^{2H(\eps ) nt}(3/4)^{0.2nt}4^{nt} F(n-2t).
$$

\end{Lemma}

\nin
(Actually we only use this with $t< O(\log n)$.)

\begin{proof}
All $G\in \mathcal{X}_2(n,t)$ can be constructed by choosing:

\mn
(i)  $K=\{x_1\dots x_t,y_1\dots y_t\}\sub V\sm K_1(G)$
(with $x_1\dots y_t$ distinct);

\mn
(ii) for each $i\in [t]$, $T_i := \Gamma(x_i)\cap\Gamma(y_i)$ of size
at most $\eps n$;

\mn
(iii) $G':= G[V']$ (including colors), where $V'=V\sm K$;

\mn
(iv) for each $v\in K$, some $Q_v\sub \Gamma(v)\sm K$ of size $s$
and colors for $\nabla(v,Q_v)$;

\mn
(v)
the remaining edges of $G$ meeting $K$
(those not in $\cup_{v\in K}\nabla(v,Q_v)$) and colors for these edges.

\mn
(The point of (iv) is that,
since $G$ is OBTF,
the colors for $\nabla(v,Q_v)$ together with
those for edges of $G'$ meeting $Q_v$ limit our choices for the remaining
edges at $v$.)

We may bound the numbers of choices in steps (i)-(iv) by $n^{2t}$,
$2^{H(\eps  )nt}$,
$F(n-2t)$, and $\left(\C{n}{s}2^s\right)^{2t}< n^{2st}$
respectively, and the number of choices for $G[K]$ (in (v))
by $\exp_3[{\C{2t}{2}}]$.

Given these choices (and aiming to bound the number of possibilities for $\nabla(K,[n]\sm K)$), we write
$\Gamma'$ for $\Gamma_{G'}$, and,
for $i\in [t]$, define:
$P_i=Q_{x_i}$, $Q_i=Q_{y_i}$;
$R_i=(\Gamma'(P_i)\cap \Gamma'(Q_i))\sm T_i$,
$R_i'=\Gamma'(P_i)\sm \Gamma'(Q_i)\sm T_i$,
$R''_i=\Gamma'(Q_i)\sm \Gamma'(P_i)\sm T_i$ and
$\bar R_i=V'\sm (R_i\cup R'_i\cup R''_i\cup T_i)$; and
$\ga_i = |R_i|$, $\ga'_i = |R'_i|$,
$\ga''_i = |R''_i|$, $\gb_i = |\bar R_i|$ and
$\gd_i =|T_i|$ ($< \eps  n)$.

We then consider (the interesting part of the argument)
the number of possibilities for
$\nabla(z,\{x_i,y_i\})$ (including colors) for
$z\in V'$.  With explanations to follow,
this number is at most:
(i) 5 if $z\in \bar R_i$;
(ii) 4 if $z\in T_i \cup R'_i\cup R''_i$; and
(iii) 3 if $z\in R_i$.  This is because:

\mn
(i) $z\not\in T_i$ excludes the four possibilities with
$z$ connected to both $x_i$ and $y_i$;

\mn
(ii) for $z\in T_i$ this is obvious; for $z\in R'_i$,
we already know the colors on some $(x_i,z)$-path of length two, so
the condition OBTF leaves only one possible color for an edge between $x_i$ and $z$,
thus excluding one of the five possibilities in (i) (and similarly for
$z\in R_i''$);

\mn
(iii) here we have (as in (ii)) one excluded color for each of $x_iz$, $y_iz$.

\mn
Thus, letting $z$ vary and noting that
$\ga_i  +\ga_i'+\ga_i''+\gb_i+\gd_i =n-2t$,
we find that the number of possibilities
for $\nabla(\{x_i,y_i\}, V')$ is at most

$$
5^{\gb_i} 4^{\ga_i'+\ga_i'' +\gd_i} 3^{\ga_i}<
4^n\left(\frac{15}{16}\right)^{\gb_i}
\left(\frac{3}{4}\right)^{\ga_i-\gb_i}
\leq 4^n
\left(\frac{3}{4}\right)^{\ga_i-\gb_i}.$$
The crucial point in all this is that $G\not\in \X_0$ guarantees
that $\ga_i-\gb_i$ is big:
each of $\ga_i +\ga_i'$, $\ga_i +\ga_i''$ is at least $0.6 n -2t -\gd_i $,
whence
$$
n-2t-\gd_i -\gb_i = \ga_i +\ga'_i+\ga''_i > 1.2n -4t -2\gd_i -\ga_i ,
$$
implying
$\ga_i -\gb_i > 0.2 n -2t -\gd_i  > (0.2-\eps ) n -2t $.
So, finally, applying this to each $i$ and combining with our
earlier bounds (for (i)-(iv) and the first part of (v)) bounds the total
number of possibilities for $G$ by
$$
2^{(H(\eps ) +o(1))nt} (3/4)^{(0.2-\eps -o(1))nt}4^{nt}F(n-2t),
$$
which is less than the bound in the lemma.

\end{proof}

\begin{Lemma}
\label{almostall}
For any
$G\in \mathcal{F}(n)\setminus\mathcal{X}_0(n)$,
$x\in V\sm (K_1(G)\cup K_2(G))$ and
$S\sub V$ of size $6s$,
there exists
 $Q_x\subset\Gamma_x\sm S$ with
\beq{VK1K2}
|Q_x| =\log n ~~\mbox{and} ~~~|V\sm  (K_1(G)\cup K_2(G)\cup \Gamma(Q_x))| < 2n^{1-\eps }.
\enq
\end{Lemma}
\begin{proof}
We have
$|(\Gamma_x\cap \Gamma_y)\sm S| \geq \eps n-6s$
for any $x,y\in V\sm (K_1(G)\cup K_2(G))$.
So, for any such $x,y$ and $Q$ a random (uniform) ($\log n$)-subset of $\Gamma_x\sm S$,
$$
\Pr(Q\cap \Gamma_y =\emptyset)
< \left(1-\frac{\log n}{n}\right)^{|\Gamma_x\cap \Gamma_y|-6s}
< \left(1-\frac{\log n}{n}\right)^{\eps  n-6s} < 2n^{-\eps }.
$$
Thus $\E |V\sm (K_1(G)\cup K_2(G)\cup \Gamma(Q))| < 2n^{1-\eps }$
and the lemma follows.
\end{proof}

For $x\in V$ and $Q_x\sub \Gamma_x$,
say $z\in V$ is {\it inconsistent for} $(x,Q_x)$ if
there is an odd-blue cycle $xx_1zx_2$
with $x_1,x_2\in Q_x$,
and write $I(x,Q_x)$ for the set of such $z$.
If in addition $y\sim x$ and $Q_y\sub \Gamma_y$,
say $z\in V$ is {\it inconsistent for}
$(x,Q_x,y,Q_y)$ if
$z\in I(x,Q_x)\cup I(y,Q_y)$ or
there is an odd-blue cycle $xx_1zy_1y$ with
$x_1\in Q_x$ and $y_1\in Q_y$,
and write $I(x,Q_x,y,Q_y)$ for the set of such $z$.

For $G\in \f(n)$, let
$K_3(G) = \{x_1\dots x_l,y_1\dots y_l\}$ be a largest possible
collection of (distinct) vertices of $V\sm (K_1(G)\cup K_2(G))$
with $x_i\sim y_i$ and
for which there exist
$Q_v\sub \Gamma_v\sm K_3$
for $v\in K_3$
satisfying (\ref{VK1K2}) and
\beq{IxQx}
|I(x_i,Q_{x_i},y_i,Q_{y_i})| >\eps  n ~\forall~i\in [l].
\enq
Set $\gk_3(G)=l$.
Let
$$
 \X_3(n,t) = \{G\in \f(n)\sm (\X_0\cup \X_1\cup \X_2):
 \gk_3(G) \geq t\}
$$
and $\X_3 = \X_3(n) = \X_3(n,s)$.

Now for $G\in \f(n)\sm \X_0$ and each
$x\in V\sm (K_1(G)\cup K_2(G))$,
fix some $Q_x\sub \Gamma_x\sm K_3(G)$ satisfying
(\ref{VK1K2}) and (\ref{IxQx}) if $x\in K_3(G)$ and (\ref{VK1K2}) otherwise.  Existence of such $Q_x$'s is given by
Lemma \ref{almostall}, and the maximality of $K_3(G)$
implies that for each $xy\in E(G- (K_1(G)\cup K_2(G)\cup K_3(G)))$
we have $I(x,Q_x,y,Q_y) \leq \eps n$.
Having fixed these $Q_x$'s,
we abbreviate $I(x,Q_x) = I(x)$ and
$I(x,Q_x,y,Q_y) = I(x,y)$.

\begin{Lemma}\label{LX5}
For sufficiently large n and $t \leq s$,
$$|\X_3(n,t)| <  (3/4)^{\eps  nt}2^{o(nt)}4^{nt} F(n-2t).
$$
\end{Lemma}

\begin{proof}
All $G\in \mathcal{X}_3(n,t)$ can be constructed by choosing:

\mn
(i) $K=\{x_1\dots x_t,y_1\dots y_t\}$
(with $x_1\dots y_t$ distinct);

\mn
(ii) $ G[V']$ (including colors), where $V'=V\sm K$;

\mn
(iii) for each $x\in K$, $Q_x$
and colors for $\nabla(x,Q_x)$;

\mn
(iv)
the remaining edges meeting $K$ and colors for these edges.

\mn
We may bound the numbers of choices in steps (i)-(iii) by $n^{2t}$,
$F(n-2t)$, and $n^{2t\log n}$
respectively, and the number of choices for $G[K]$
(in (iv))
by $3^{\C{2t}{2}}$.
Notice that the choices in (i)-(iii) determine the sets $I(x_i,y_i)$, which
in particular are of size at least $\eps n$.

As in Lemma \ref{LX3},
the interesting point is
the number of possibilities for $\nabla(z,\{x_i,y_i\})$
for $z\in V'$.
In general, if $z\in \Gamma(Q_{x_i})\cap\Gamma(Q_{y_i})$
this number is at most 4,
since (because $G$ is to be OBTF)
any path $x_ixz$ with $x\in \Gamma(Q_{x_i})$---so we
already know the colors of $x_ix$ and $xz$---excludes
one possible color for a (possible) edge $x_iz$,
and similarly for $y_i$.
Moreover, if $z\in I(x_i,y_i)$ then the number is at most 3:
if $z\in I(x_i)$ then an edge $x_iz$ of {\em either} color
gives an odd-blue triangle,
and similarly if $z\in I(y_i)$; and
otherwise, we cannot have $z$ joined to both $x_i$ and $y_i$
without creating an odd-blue triangle
(and we already know an edge $x_iz$ or $y_iz$
admits at most one possible color).
If $z\not\in \Gamma(Q_{x_i})\cap\Gamma(Q_{y_i})$, then
we just bound the number by 9,
noting that the number of such $z$ is
$o(n)$ (since the $Q$'s satisfy (\ref{VK1K2})).

Thus
the number of possibilities
for $\nabla(\{x_i,y_i\}, V')$ is at most

$$
4^{n-\eps  n}3^{\eps  n}9^{o(n)} =
4^n
(3/4)^{\eps  n}2^{o(n)};$$
so combining with our earlier bounds we find that the number
of possibilities for $G$ is less than
$
(3/4)^{\eps  nt}2^{o(nt)}4^{nt}
F(n-2t).$
\end{proof}

For $G\not\in \X_0$
let $K(G) = K_1(G)\cup K_2(G)\cup K_3(G)$.
As we will see,
Theorem \ref{MT} is now an easy consequence of
\begin{Lemma}\label{lbb}
For each $G\in \f(n)\sm (\X_0\cup\cdots \cup\X_3)$, $G-K(G)$ is blue-bipartite.
\end{Lemma}

\begin{proof}
We first assert that
\beq{short}
\mbox{$G-K(G)$ contains no odd-blue cycle of
length 4 or 5.}
\enq
To see this, suppose
$x_1\dots x_q$ is a cycle in $G':=G-K(G)$ with $q\in\{4,5\}$,
and (with subscripts taken mod $q$) let
$$z\in \bigcap_{i=1}^q\Gamma(Q_{x_i})\sm (\bigcup_{i=1}^qI(x_i,x_{i+1})
\cup \{x_1\dots x_q\}).$$
(Note that there is such a $z$; in fact the size of the set on the r.h.s. is at least
$n - |K_1(G)\cup K_2(G)\cup K_3(G)|-q\log n-2qn^{1-\eps } -q\eps  n -q$,
so essentially
$(1-q\eps )n$.)

Let $w_i \in \Gamma(z)\cap Q_{x_i}$ ($i\in [q]$).
Each of the closed walks $zw_ix_ix_{i+1}w_{i+1}z$ is even-blue,
either (in case it is a 5-cycle) because $z\not\in I(x_i,x_{i+1})$,
or (otherwise) because $G$ is OBTF, where we use the easy
\begin{eqnarray}\label{short2}
\mbox{any non-simple closed walk of length at most 5 }
~~~~~~~~~~~~~~~ \nonumber\\
~~~~~~~~~~~~~~~\mbox{in an OBTF graph is even-blue.}
\end{eqnarray}
But since these walks together with the original cycle use each edge
of $G$ an even number of times, it follows that the original cycle is also
even-blue.

\bigskip
We now define the blue-bipartition for $G'$ in the natural way.
Note that the diameter of $G'$ is at most 2
(in fact any two vertices of $G'$ have at least $\eps  n-o(n)$ common
neighbors), and that (\ref{short}) and (\ref{short2}) imply
that for any two vertices $x,y$, all $(x,y)$-paths of length
at most 2 have the same {\em blue-parity} (defined in the obvious way).
We may thus fix some vertex $x$ and let $U$ consist of those vertices
for which this common parity is even (so $x\in U$) and $W=V(G')\sm U$.
That this is indeed a blue-bipartition is again an easy consequence of
(\ref{short}) and (\ref{short2}).

\end{proof}

\nin
{\em Proof of Theorem} \ref{MT}.
For $G\in \f(n)\sm \X_0$, Lemma \ref{lbb} gives $\gk(G)\leq \gk_1(G)+2(\gk_2(G)+\gk_3(G))$,
so that $\gk(G) \geq t$ implies that either
$\gk_1(G)\geq t/3$ or at least one of $\gk_2(G)$, $\gk_3(G)$ is at least $t/6$.
It follows that
$$
\{G:\gk(G)\geq t\} \sub \X_0\cup \X_1(n,\lceil t/3\rceil)
\cup \X_2(n,\lceil t/6\rceil)\cup \X_3(n,\lceil t/6\rceil)
$$
(since for $G\not\in \X_0$:
$\kappa_1(G)\geq t/3 \Ra G\in \X_1(n,\lceil t/3\rceil)$;
$\kappa_2(G)\geq t/6 \Ra G\in \X_1\cup\X_2(n,\lceil t/6\rceil)$;
$\kappa_3(G)\geq t/6 \Ra G\in \X_1\cup\X_2\cup \X_3(n,\lceil t/6\rceil);$
and
$\X_1\sub \X_1(n,\lceil t/3\rceil)$,
$\X_2\sub \X_2(n,\lceil t/6\rceil)$).

The theorem, with any (fixed, positive) $c < 1-0.6\log_23$ 
satisfying
\beq{c}
2^{-c} > 
2^{H(\eps)+\eps -1}+
2^{H(\eps)}(3/4)^{0.1}+(3/4)^{\eps/2},
\enq
now follows from
Lemmas \ref{Xdoi}-\ref{LX3} and \ref{LX5}.
\qed

From this point we set
$b(n)= 2^{\C{n+1}{2}-1}$ ($\sim |\B(n)|$).

\begin{proof}[Proof of Theorem \ref{OBTFthm}]

We prove Theorem \ref{OBTFthm} by showing by induction that, for some
constant $\gD$, $c$ as in Theorem \ref{MT} and all $n$,

\begin{equation}
\label{FnBn}
F(n)\leq (1+\Delta\cdot 2^{-cn})b(n)
\end{equation}
To see this,
choose $n_0$ large enough so that the previous results in this section
are valid for $n\geq n_0$, and then
choose $\Delta >2$ (say)
so that (\ref{FnBn}) holds for $n\leq n_0$.
Assuming
(\ref{FnBn}) holds for $n-1$, we have, using Theorem \ref{MT} for the first
inequality,
\begin{eqnarray}
|\f(n)\sm \B(n)|&=&|\{G:\gk(G)>0\}|\nonumber\\
&<&
2^{(1-c)sn}F(n-s)+ 2^{n-cn}F(n-1) \nonumber\\
&<&
2^{(1-c)sn}(1+\gD 2^{-c(n-s)})   b(n-s)\nonumber\\
&&~~~~~~~~~~~~~~~~~~~~+
 2^{n-cn}(1+\gD 2^{-c(n-1)})   b(n-1)\nonumber\\
&=&
[2^{\C{s}{2}-csn}(1+\gD 2^{-c(n-s)})\nonumber\\
&&~~~~~~~~~~~~~~~~~~~~+ 
2^{-cn}(1+\gD 2^{-c(n-1)}) ]  b(n).\label{cc}
\end{eqnarray}
So, since $|\B(n)| <b(n)$ and
the coefficient of $b(n)$ in (\ref{cc}) is less
than $\gD 2^{-cn}$, we have (\ref{FnBn}).
\end{proof}

Feeding this back into Theorem \ref{MT}
we obtain a quantitative
strengthening of Theorem \ref{OBTFthm} that will be useful below.
(Recall we assume $G\in \f(n)$.)

\begin{thm}\label{MC}
For any constant $\gd<c/3$, $t\leq s$ and large enough
n,
\beq{corbd}
|\{G:\gk(G)\geq t\}| <  2^{-\gd nt}b(n).
\enq
\end{thm}
\nin
{\em Proof.}
We have (for large enough $n$)
\begin{eqnarray*}
|\{G:\gk(G)\geq t\}| &<&
2^{(1-c)sn}F(n-s)+
2^{(1-c)n\lceil t/3\rceil}F(n-\lceil t/3\rceil)
\\
&<&
2[2^{(1-c)sn}b(n-s)+
2^{(1-c)n\lceil t/3\rceil}\cdot  b(n-\lceil t/3\rceil)]\\
&<&2^{-\gd nt}b(n).
\end{eqnarray*}\qed

In what follows we will also need an analogue of $\gk$
for
edge removals,
say
$$\gc(G) := \min\{|E'|:\mbox{$E'\sub E(G), G-E' $ BB}\}.$$
\begin{Lemma}\label{LE'}
There is a constant $C'$ such that, for
sufficiently large n,
$$|\{G:\gc(G) > C'\sqrt{n}\log^{3/2}n\}| < n^{-3n} b(n).$$
\end{Lemma}
\nin

\begin{proof}
Fix $A> ((12\log_23)/c +3)^{1/2}$.
The story here is that $\gk(G)$ small implies $\gc(G)$ small unless
we encounter the following pathological situation.
Let $\Y(n) $ consist of those $G\in \f(n)$ for which there is
some $K\sub V$ of size at most $k:=(12\log n)/c$
such that $G-K$ is BB and there are disjoint $S,T\sub V\sm K$,
each of size $A\sqrt{n\log n}$, with $\nabla_G(S,T)=\0$.
We assert that, for any constant $C''<A^2 -(12\log_23)/c$ 
(and large $n$),

\beq{Y(n)}
|\Y(n)| < \exp_2[\C{n}{2} -C''n\log n].
\enq
This is a routine calculation:
the number of choices for $G\in \Y(n)$ is at most
$$
3^n\exp_3[\C{k}{2}+k(n-k)]\exp_2[\C{n}{2}-A^2n\log n],
$$
where the first term corresponds to the choices of 
$K$, the blue-bipartition and $S,T$;
the second to edges of $G$ meeting $K$;
and the third to the remaining edges
(whose colors are determined by the blue-bipartition).
This gives (\ref{Y(n)}).

\medskip
Thus, in view of
Theorem \ref{MC} (noting $(12/c)\log n < s$; see (\ref{C})),
Lemma \ref{LE'} will follow from
\beq{gbgc}
G\in \f(n)\sm \Y(n), ~\gk(G) < (12/c)\log n ~~\Ra ~~
\gc(G) < C'\sqrt{n}\log^{3/2} n
\enq
(for a suitable $C'$).
To see this, suppose $G\not\in \Y(n)$
and $G-K$ is BB with
$|K| < (12/c)\log n$.
Let $X\cup Y$ be a blue-bipartition of $G-K$,
and write $R$ and $B$ for the sets of red and blue edges of $G$.
Given
$x\in K$, let
$R_X=R_X(x)= \{v\in X: xv\in R\}$,
and define $B_X,R_Y,B_Y$ similarly.
Then $G$ OBTF implies
$$
\nabla(R_X,B_X) =\nabla(R_X,R_Y) =\nabla(B_X,B_Y) =\nabla(R_Y,B_Y) =\0,$$
whence (since $G\not\in \Y(n)$)
WMA that at most two of
$R_X,B_X,R_Y,B_Y$
have size at least $A\sqrt{n\log n}$,
and if exactly two then these must be either
$R_X$ and $B_Y$, or $B_X$ and $R_Y$.
Thus there is a set $E'(x)$ of at most $2A\sqrt{n\log n}$
edges at $x$
so that either $\nabla(x,X)\sm E'(x)\sub R$ and
$\nabla(x,Y)\sm E'(x)\sub B$
or vice versa.
Setting $E' = E(K)\cup \bigcup\{E'(x):x\in K\}$,
we find that $G-E'$ is BB with
$|E'| <  \C{|K|}{2} +2A|K|\sqrt{n\log n}  <
C'\sqrt{n}\log^{3/2} n,$
for any $C' > 24A/c$ (and large $n$).
\end{proof}

\section{Blue-bipartite graphs}\label{SBB}

We continue to assume $G\in \f(n)$
and now need some understanding of the sizes of the sets
$$
\pee(G) := \{P\in \pee(n): G(P)=G\}
$$
(see following (\ref{SATe2}) for $\pee(n)$ and $G(P)$).
Recall (see property (e) of $G(P)$) that $G(P)$ determines the cover graph
of $P$; thus,
as observed in \cite{BBL}, we trivially have
\begin{equation}
\label{trivialUB}
|\pee(G)|<(2n)! <n^{2n}  ~~~~~\forall G\in \mathcal{F}(n),
\end{equation}
since a poset is determined by its cover graph and any one of its
linear extensions.

If $P\in \pee(G)$ then the cover graph of $P$ is
$C(G)$, defined to be the graph on $\{x_1\dots x_n,\bar{x}_1\dots\bar{x}_n\}$
with, for each $w_iw_j\in E(G)$,
edges $x_ix_j$ and $\ov{x}_i\ov{x}_j$ if $w_iw_j$ is blue, and
$x_i\ov{x}_j$ and $\ov{x}_ix_j$ if it is red.
By property (c) in the definition of $\mathcal{P}(n)$,
the orientation of either of the edges of $C(G)$ corresponding to
a given edge of $G$ determines the orientation of the other; so
we speak, a little abusively, of orienting the edges of $G$.

A basic observation is that the orientations of the
edges of any triangle $w_iw_jw_k$ of $G$,
are determined by the orientation of any one of them.
Suppose for instance (other cases are similar)
that the edges of $w_iw_jw_k$ are all {\em red}, and that $x_i<\bar{x_j}$
(so also $x_j<\bar{x_i}$).
We must then have $\bar{x_k} > x_i,x_j$
(and $x_k < \bar{x_i},\bar{x_j}$), since (e.g.)
$\bar{x_k}< x_i$ would imply $x_k > \bar{x_i}$,
and then $x_k > \bar{x_j}$ would give $x_k > \bar{x_k}$,
while $x_k < \bar{x_j}$ would give $\bar{x_j} > x_j$,
in either case a contradiction.
It follows that the orientation of either of
$e,f\in E(G)$ determines the orientation of the other
whenever there is a sequence
$T_0,\ldots,T_l$ of triangles with $e$ (resp. $f$)
an edge of $T_1$ (resp. $T_l$) and
$T_{i-1},T_i$ sharing an edge for each $i\in [l]$.
We then write $e\equiv f$, and call the classes of this equivalence relation
{\em triangle-components} of $G$.  If there is just one equivalence class,
we say $G$ is {\em triangle-connected}.

In general the preceding discussion bounds $|\pee(G)|$ by $2^{\eta(G)}$ with
$\eta(G)$ the number of triangle components of $G$;
but all we need from this is

\begin{Lemma}\label{L2ors}
If $G\in \B(n)$ is triangle-connected then $|\pee(G)|\leq 2$.
\end{Lemma}
\nin

\nin
(Actually it's easy to see that equality holds.)
The last piece needed for the proof of Theorem \ref{SATthm} is

\begin{Lemma}\label{Llast}
There are at most $2^{-\gO(n)} b(n)$
$P\in \pee(n)$ with $G(P)$ in $\B(n)$ and not triangle-connected.
\end{Lemma}

\begin{proof}

Fix $\gd >0 $ with $5(1-H(\gd)) > 3\gd$,
and for $G\in \B(n)$ let
$X(G) =\{x\in V:d_G(x) < \gd n\}$.  Set $D=5/\gd$.
We first dispose of some pathologies:

\begin{prop}\label{Ppath}
All but
at most
$
n^{-3n}b(n)
$
$G\in \B(n)$ satisfy

\mn
{\rm (i)}  $|X(G)|< D\log n$;

\mn
{\rm (ii)}  $\not\exists $ disjoint $Y,Z\sub V$ with $|Y||Z| = Dn\log n$ and
$\nabla(Y,Z)=\0$;

\mn
{\rm (iii)}
$\forall x\in V\sm X(G)$, the size of the largest connected component of $G[\gG_x]$
is at least $d_G(x)- D\log n$;

\mn
{\rm (iv)}
$\forall x\in V\sm X(G)$,
$|\{y\in V\sm X(G): |\gG_x\cap \gG_y|< D\sqrt{n\log n}\}|< D\log n.$
\end{prop}

\begin{proof}
(i)  We may specify $G\in \B(n)$ violating (i) by choosing:
a blue-bipartition $S\cup T$; $X'\sub X(G)$ of size $D\log n$;
$E(X')\cup \nabla(X',V\sm X')$; and $G-X'$.
The numbers of ways to make these choices are at most:
$2^n$; $\C{n}{D\log n}$;
$(\sum_{i< \gd n}\C{n}{i})^{D\log n}$;
and $\exp_2[\C{n-D\log n}{2}]$; and,
in view of
our restriction on $\gd$, the product of these bounds is
much less than
$n^{-3n}b(n)$.

\mn
(ii)
For use in (iii) we show a slightly stronger version, say
 (ii$'$), which is just (ii) with $D$ replaced by 4.
We may specify $G\in \B(n)$ violating (ii$'$)
by choosing a blue-bipartition $S\cup T$ and $Y,Z$
in at most
(say) $5^n$ ways,
and then the edges of $G$ in at most
$\exp_2[\C{n}{2} -4n\log n]$ ways.

\mn
(iii)
Here we simply observe that any $G\in \B(n)$ satisfying (ii$'$)
also satisfies (iii).
To see this,
notice that if $d_G(x) \geq \gd n$ and $G$ satisfies (ii$'$), then
there is no $K\sub \gG_x$ with $|K|\in (D\log n,d_G(x)-D\log n)$ and
$\nabla(K,\gG_x\sm K)=\0$.
But then if (iii) fails at $x$, it must be that all components
of $G[\gG_x]$ have size less than $D\log n$, in which case
we get the supposedly nonexistent
$K$ as a union of components.

\mn
(iv)
Here, with $k=D\log n$, we may specify a violator by choosing:
a blue-bipartition $S\cup T$; $x\in V$; $\gG_x$ of size at least $\gd n$;
$y_1\dots y_k\in V$;
$\gG_{y_i}\cap \gG_x$ of size at most $r:= D\sqrt{n\log n}$ (for $i\in [k]$);
and $E(G)\sm (\nabla(x)\cup \nabla(\{y_1\dots y_k\}),\gG_x)$.
The number of possibilities for this whole procedure is at most
$$
\mbox{$2^n\cdot n\cdot 2^n\cdot n^k\cdot
\max_{m\geq\gd n}\{(2 \C{m}{r})^k\exp_2 [\C{n}{2}-km +k^2]\}.$}
$$
(We used $\sum\{\C{m}{i}: i\leq r\}< 2\C{m}{r}$; the irrelevant $k^2$ allows
for some $y_i$'s in $\gG_x$;
of course we could have strengthened $D\sqrt{n\log n}$ to some $\Omega(n)$.)

\end{proof}

We now return to the proof of Lemma \ref{Llast}.
Let $\h(n)$ consist of those $G\in \B(n)$ that are not triangle-connected
and for which (i)-(iv) of Proposition \ref{Ppath} hold.
The proposition and (\ref{trivialUB}) imply that Lemma \ref{Llast} will follow from
\beq{H(n)}
\sum\{|\pee(G)|:G\in \h(n)\} < 2^{-\Omega(n)}b(n).
\enq

Temporarily fix $G\in \h(n)$ and set $X=X(G)$
and $W=W(G)=V\sm X$.
For $x\in W$ let $\gG_x'=\gG_x\cap W$.
Let $L_x$ be the intersection of
(the vertex set of) the largest connected
component of $G[\gG_x]$ with $\gG_x'$, $K_x = \gG_x'\sm L_x$, and
$E_x=\nabla(x,K_x)$, and
observe that
all edges contained in $\{x\}\cup L_x$ lie in the
same triangle component of $G$, say $\cee(x)$.

For $x,y\in W$, write $x\lra y$
if
$E(L_x\cap L_y)\neq\0$, and note this implies $\cee(x)=\cee(y)$.
By (ii)
we have $x\lra y$ whenever
$|L_x\cap L_y|> 2\sqrt{Dn\log n}$,
whence, by (iv) and (i),
$$ |\{y\in W:y\hspace{.01in}\not\hspace{-.015in}
\lra x\}| < D\log n ~~~~\forall x\in W.$$
In particular, ``$\lra$" is the edge set of a connected graph
on $W$, implying all triangle components $\cee(x)$
are
the same; that is, $E(W)-\cup\{E_x:x\in W\}$ is 
contained in a single triangle-component of $G$.
Note also that $z\in K_x$ implies $|\gG_x'\cap \gG_z'| <|K_x| < D\log n$
(by (iii)),
so that, again using (iv) and (i), we have
\beq{fewx}
\mbox{$ |\{x\in W:z\in K_x\}| < \min\{D\log n,\sum_{x\in W} |K_x|+1\} ~~~~\forall z\in W.$}
 \enq
(The extra 1 in the trivial second bound will sometimes save us from dividing
by zero.)
In what follows we set $\xx=|X|$,
$m=|W|$ ($= n-\xx$),
$k_x=|K_x|$ and
$\uk =(k_x:x\in W) \in [0,D\log n]^W$.

We now consider the sum in (\ref{H(n)}), i.e. the number of ways
to choose $G\in \h(n)$ and $P\in \pee(G)$.
As usual there are $2^{n-1}$ ways to choose the blue-bipartition.
We then choose $X=X(G)$ and the edges meeting $X$,
the number
of ways to do this
for a given
$\xx$
being
at most $\C{n}{\xx}\C{n}{<\gd n}^t<\exp_2[(\log n + H(\gd)n)t]$,
define $W$ and $\gG_x'$ as above, and let $H=G[W]$.
{\em Vertices discussed from this point are assumed to lie in $W$},
and we set $d_x'=d_H(x)$.

We first consider a fixed $\uk$, setting
$g(\uk) =\min\{D\log n,\sum k_x+1\}$.
There are at most $\prod \C{m}{k_x} < \exp_2[\sum k_x\log n]$ ways to choose
the sets $K_x $.  Once these have been chosen, we write
$\h$ for the set of possibilities remaining for $H$.
For a particular $H\in \h$,
let $\U_H = \{\{y,z\}:\exists x ~ y\in K_x,z\in L_x\}$.
By (\ref{fewx}) we have
\beq{UH}
|\U_H| > \frac{1}{g(\uk)} \sum_x(d_x'-k_x)k_x > \frac{\gd n}{2g(\uk)}\sum k_x.
\enq

Given an ordering $\gs =(x_1\dots x_m)$ of $W$, we specify $H$ by choosing,
for $i= 1\dots m-1$,
$\nabla (x_i,\{x_{i+1}\dots x_m\}\sm K_{x_i})$.
Note that if $i<j,l$, and exactly one of $x_j,x_l$ belongs to
each of $K_{x_i},L_{x_i}$, then
$x_j\not\sim x_l$ is established in the processing of $x_i$,
so we never need to consider potential edge $x_jx_l$ directly.
Thus the number of choices, say $f(\gs,H)$,
that we actually make in producing
a specific $H$ is at most
$$\C{m}{2} - |\{(i,\{j,l\}):i<j,l ; x_j\in K_{x_i}, x_l\in L_{x_i}\}|.$$
For a fixed $H$ and random (uniform) $\gs$, the expectation of the
subtracted expression is at least $|\U_H|/3 $.
This gives (using (\ref{UH}))
\begin{eqnarray}\label{sigH}
\frac{1}{m!}\sum_{\gs}\sum_Hf(\gs,H) &=&
\sum_H\frac{1}{m!}\sum_{\gs}f(\gs,H) \nonumber\\
&<&
\left( \C{m}{2} - \frac{\gd n}{6g(\uk)}\sum k_x\right)|\h|.
\end{eqnarray}
Thus there is some $\gs$ for which
$\sum_Hf(\gs,H)$ is at most the r.h.s. of (\ref{sigH}),
whence, we assert,
$$|\h| < \exp_2[\C{m}{2} - \frac{\gd n}{6g(\uk)}\sum k_x].$$
{\em Proof.}
This is a standard observation:
for a given $\gs$ we may think of the above procedure as
a decision tree, with $f(\gs,H)$ the length of the path
leading to the leaf $H$;
and we then have
$$
1\geq \sum_H 2^{-f(\gs,H)} \geq |\h| \exp_2[-|\h|^{-1}\sum_Hf(\gs,H)].
$$\qqed

Finally, we need to choose an orientation.
By Lemma \ref{L2ors}
there are just two ways to orient the edges of the 
triangle component of $G$ containing
$ H-\cup\{E_x:x\in W\}$.
We then extend to $\cup\{E_x:x\in W\}$ and the remaining edges meeting $X$ 
in at most $\exp_2[\gd nt + \sum\{k_x:x\in W\}]$ ways.
In summary the number of ways to choose the pair $(G,P)$ is less than
\beq{big1}
2^n\sum_{\xx}\sum_{\uk}\exp_2[\C{m}{2} + ((H(\gd)+\gd)n +\log n)\xx
+ (1+\log n -\frac{\gd n}{6g(\uk)})\sum_{x\in W} k_x],
\enq
with the double sum over $\xx\in [0,D\log n]$ and $\uk\in [0,D\log n]^m$,
excluding the $(0,\underline{0})$-term, which counts
only triangle-connected graphs.
Noting that
$\C{m}{2} =\C{n}{2} -\xx(n-\xx)-\C{\xx}{2}$,
we find that, for any constant $\gc <1- H(\gd)-\gd$,
the expression in (\ref{big1}) is (for large $n$) less than
$$
2^{\C{n+1}{2}}
\sum_{\xx}\sum_{\uk}\exp_2[-\gc n\xx - \frac{\gd n}{7 g(\uk)}\sum k_x]
<2^{-\Omega (n)} b(n).
$$

\end{proof}

\section{Proof of Theorem \ref{SATthm}}

This is now easy.
We have
\[
|\mathcal{P}(n)|= |\cup\{\pee(G):G\in \B(n)\}|+
|\cup\{\pee(G):G\in \f(n)\sm\B(n)\}|
\]
Here the first term on the r.h.s. is asymptotic to $2^{\C{n+1}{2}}$ by (\ref{BB}) and
Lemmas \ref{L2ors} and \ref{Llast}; so we just need to show
that the second is $o(b(n))$.
Moreover, according to Lemma \ref{LE'} and (\ref{trivialUB}), it's
enough to show this when we restrict to $G$ with $\gc(G) \leq C'\sqrt{n}\log^{3/2}n$ ($C'$ as in Lemma \ref{LE'}).
Thus the theorem will follow from
\beq{last}
\sum
\{|\pee(G)|:G\in \f(n)\sm\B(n),\gc(G) \leq  C'\sqrt{n}\log^{3/2}n\}
<2^{-\gO (n)}b(n).
\enq
{\em Proof.}
For $G$ as in (\ref{last}) let $E' = E'(G)$ be a subset of $E(G)$ of
size at most $C'\sqrt{n}\log^{3/2}n$ with $G-E'$ BB.
To bound the sum in (\ref{last})---i.e. the
number of possibilities for a pair $(G,P)$ with $G$ as in (\ref{last})
and $P\in \pee(G)$---we consider two cases
(in each of which we use the fact
that if $P\in \pee(G)$,
then the poset generated by the
restriction of $P$ to $E(G)\sm E'$ belongs to $\pee(G-E')$).

For $G-E'$ not triangle-connected,
we may think of choosing $G-E'$ and $P'\in \pee(G-E')$,
which by Lemma \ref{Llast} can be done in at most $2^{-\Omega(n)} b(n)$ ways,
and then choosing $E'$ and extending $P'$ to $P\in \pee(G)$
(that is, choosing orientations for the edges of $E'$),
which can be done in at most $2^{o(n)}$ ways.

For $G-E'$ triangle-connected we specify $(G,P)$
by choosing:
$G$; $E'$; $P'\in \pee(G-E')$; and
$P$ extending $P'$ to $E(G)$.  The number of possibilities in the first
step is at most $2^{-\Omega( n)}b(n)$ by Theorem \ref{MC};
the numbers of possibilities in the
second and fourth steps are $2^{o(n)}$; and there are (by Lemma \ref{L2ors})
just two possibilities in step 3.\qed

\section{Questions}

One obvious question suggested in \cite{BBL} is
estimation of the number of $k$-SAT functions for other values of $k$.
Here fixed $k$ seems to us most interesting.
It is conjectured in \cite{BBL} that in this case the number of
$k$-SAT functions is $\exp_2[(1+o(1))\C{n}{k}]$, and we see no
reason not to expect
\begin{conj}
For fixed k the number of k-SAT functions of n variables
is asymptotically $\exp_2[n +\C{n}{k}]$.
\end{conj}

We would also like to mention one natural question
that, surprisingly, seems not to have been considered previously:
\begin{question}
How many posets can have the same n-vertex cover graph?
\end{question}

\nin
Though it doesn't even seem obvious that the answer here
is $2^{\omega(n)}$, a construction of Graham Brightwell 
\cite{Brightwell}
gives a lower bound $(c\log n/\log\log n)^n$.

\bn
{\bf Acknowledgment}  
Thanks to the referee for a careful reading.

\bigskip
\begin{flushleft}
Department of Mathematics\\
Rutgers University\\
Piscataway NJ 08854 USA\\
ilinca@math.rutgers.edu\\
jkahn@math.rutgers.edu
\end{flushleft}

\end{document}